\documentclass[12pt]{amsart}
\usepackage{mathrsfs}
\usepackage{amsfonts}
\usepackage{amsmath}
\usepackage{amssymb}
\usepackage{graphicx}
\usepackage{color}
\usepackage{xcolor}
\usepackage{amsthm,amscd}
\usepackage{calc}
\usepackage{hyperref}
\usepackage{setspace}

\newtheorem{thm}{Theorem}[section]
\newtheorem{cor}[thm]{Corollary}
\newtheorem{lem}[thm]{Lemma}
\newtheorem{prop}[thm]{Proposition}

\newcommand{\n}{\nabla}

\newcommand{\ib}{\bar{i}}
\newcommand{\jb}{\bar{j}}
\newcommand{\kb}{\bar{k}}
\newcommand{\lb}{\bar{l}}
\newcommand{\ub}{\bar{u}}

\newcommand{\w}{\omega}

\newcommand{\vn}{\vec{n}}

\newcommand{\pd}{\partial}

\newtheorem{theorem}{Theorem}

\newtheorem{question}[theorem]{Question}

\title[]{The dimension of polynomial growth holomorphic functions and forms on gradient K\"ahler Ricci shrinkers}

\author{Fei He}

\author{Jianyu Ou}

\email{hefei@xmu.edu.cn and oujianyu@xmu.edu.cn}
\address{School of Mathematical Science, Xiamen University, 422 S. Siming Rd. Xiamen, Fujian, P.R.China, 361005.}

\begin{document}

\maketitle

\begin{abstract}
    We study polynomial growth holomorphic functions and forms on complete gradient shrinking Ricci solitons. By relating to the spectral data of the $f$-Laplacian, we show that the dimension of the space of polynomial growth holomorphic functions or holomorphic $(p,0)$-forms are finite. In particular, a sharp dimension estimate for the space of linear growth holomorphic functions was obtained. Under some additional curvature assumption, we prove {\color{black} an almost} sharp estimate for the frequency of polynomial growth holomorphic functions, which was used to obtain dimension upper bound as a power function of the polynomial order.
\end{abstract}

\section{Introduction}

In 1975, S. T. Yau 
{{\cite{Yau1}}} generalized the classic Liouville theorem on $\mathbb{R}^n$ to complete Riemannian manifolds with nonnegative Ricci curvature, their seminal work has inspired a lot of study on harmonic and holomorphic functions on curved spaces.  
In particular, denote by $\mathcal{H}_d(M)$ the space of polynomial growth harmonic functions of at most order $d$ on $M$,
Yau conjectured {\cite{Yau2}} that $dim \mathcal{H}_d < \infty$ if the Ricci curvature is nonnegative. 
This conjecture was completely proved by T. H. Colding and W. P. Minicozzi {\cite{CM1,CM1997}}. 
P. Li gave a short proof of Yau's conjecture in {\cite{Li1997}} under weaker assumptions, where it was shown that 
$\mathcal{H}_d(M)\leq c(n)d^{n-1},$
the power $n-1$ is sharp, see also {\cite{CM1998}} for another proof.
In fact, the study of $\mathcal H_d(M)$ has been further developed by many authors, see for e.g. \cite{LW1999, Xu,Ding,Huang,HH22}.

For complex manifolds, it is a classic problem to deduce information on the complex structure from curvature assumptions, hence it is natural to study the space of holomorphic functions and holomorphic forms. 
In the K\"ahler case, let $\mathcal{O}_d(M)$ be the space of polynomial growth holomophic functions of at most order $d$. 
L. Ni {\cite{Ni04}} proved the sharp estimate $dim(\mathcal{O}_d(M^n))\leq dim(\mathcal{O}_d(\mathbb{C}^n))$ on complete K\"ahler manifold with non-negative bisectional curvature and Euclidean volume growth, then  B. Chen, X. Fu, L. Yin and X. Zhu {\cite{CFYZ}} removed the Euclidean volume growth condition.
Later, G. Liu {\cite{Liu16}} proved the sharp dimension estimate when the holomorphic sectional curvature is nonnegative.
Recently, G. Liu {\cite{Liu21}} proved that $dim \mathcal{O}_d \leq c d^n$ on K\"ahler manifolds with non-negative Ricci curvature and Euclidean volume growth.

In this paper, we consider holomorphic functions and forms on gradient shrinking K\"ahler Ricci solitons. Recall that a gradient shrinking Ricci soliton (or `Ricci shrinker' for short) is a complete Riemannian manifold $(M, g)$ with a smooth potential function $f\in C^{\infty}(M)$, satisfying
$$Ric + \n\n f = \lambda g,$$
for some constant $\lambda > 0$ (usually taken to be $\frac{1}{2}$), where $Ric$ denotes the Ricci curvature and $\n \n f$ the Hessian of $f$.
For K\"ahler gradient Ricci shrinkers, the potential function $f$ satisfies the additional condition that $f_{ij} = f_{\ib \jb} = 0$ in complex coordinates.

Perhaps the most prominent role of gradient Ricci shrinkers are as singularity models for the Ricci flow. It is known that finite-time singularities in dimension three and type-I singularites of the Ricci flow are modeled on gradient Ricci shrinkers \cite{Nab10}, \cite{EMT11}. Recently, Bamler proved that in higher dimensions, blow-up limits for finite-time singularties of the Ricci flow are so called metric solitons, see \cite{Bam20c} \cite{Bam2023} for details.
Ricci shrinkers have been classified for (real) dimension $2$ and $3$. Recently, a classification of complex $2$-dimensional gradient K\"ahler Ricci shrinkers is given by the combined works of \cite{CDS19}, \cite{BCCD} and {\cite{LW23}}. 
However, classification of Ricci shrinkers in general dimensions seems out of reach at this moment. 
Nevertheless, extensive studies on Ricci shrinkers have revealed many properties of their geometry, in particular, they share some surprising similarities with Ricci-nonnegative spaces (see for e.g. {\cite{Cao10}}), hence it is natural to ask the following analogue of Yau's conjecture.

\begin{question}\label{Yau_s}
    (a) Let $(M,g, f)$ be a complete gradient Ricci shrinker, is $\dim \mathcal H_d(M)<\infty$ ? \\
    (b) Let $(M,g, f)$ be a complete gradient K\"ahler Ricci shrinker, is $\dim \mathcal{O}_d(M) < \infty$?
\end{question}

This question has attracted more attention recently. Till now, only partial confirmations are know for $\mathcal{H}_d$, see the work of J.-Y. Wu and P. Wu {\cite{JWPW}} for the case when the scalar curvature $S$ has quadratic decay, and see the work of W. Mai and J. Ou {\cite{MO22}} for the case when $S$ is constant. For $\mathcal{O}_d$, this question has been answered affirmatively by O. Munteanu and J. Wang {\cite{MW2014}}.

Our first main theorem is based on the spectral theory for the $f$-Laplacian on gradient Ricci shrinkers. The $f$-Laplacian is self-adjoint with respect to the weighted measure $e^{-f} dv$. It follows from the Bakry-Emery log-Sobolev inequality that the $f$-Laplacian has discrete spectrum, and the first non-zero eigenvalue satisfies $\lambda_1(\Delta_f) \geq \frac{1}{2}$, this was shown by Hein and Naber \cite{HN2014}. Cheng and Zhou {\cite{CZ2016}} further proved that if $\lambda_1 = \frac{1}{2}$, then the soliton must split as a product. Based on an observation that polynomial growth holomorphic functions has an expansion by eigenfunctions of $\Delta_f$, we prove the following:

\begin{thm}\label{thm: main theorem 1}
   Let $(M, g, f)$ be a complete gradient K\"ahler Ricci shrinker of complex dimension $m$, then the followings hold. 
   \begin{itemize}
       \item[(a)] Let $0 = \tilde{\lambda}_0 <  \tilde{\lambda}_1< ...< \tilde{\lambda}_s \leq \frac{d}{2}$ be the distinct eigenvalues of $\Delta_f$ no bigger than $\frac{d}{2}$, and denote by $mult(\tilde{\lambda}_i)$ the multiplicity of $\tilde{\lambda}_i$. Then
       \[
        \dim_{\mathbb{C}} \mathcal{O}_d \leq 1 + \frac{1}{2} \sum_{i = 1}^s mult(\tilde{\lambda}_i).
       \]
       \item[(b)]    For each $u \in \mathcal{O}_d$, there exist holomoprhic functions 
         $\tilde{u}_i \in \mathcal{O}_{2\tilde{\lambda}_i}$, $i = 1,2,...,s$, such that
         \[
         u  = \tilde{u}_0+ \tilde{u}_1+  ... + \tilde{u}_s,
         \]
         where $\tilde{u}_i$ is an eigen-function of $\Delta_f$ (hence of the Lie derivative $\mathcal{L}_{\n f }$) with respect to the eigenvalue $\tilde{\lambda}_i$.
       \item[(c)] We have a sharp dimension estimate for linear growth holomorphic functions, \[
         \dim_{\mathbb{C}}\mathcal{O}_1 \leq m+1,
         \]
         which achieves equality only on the Gaussian soliton on $\mathbb{C}^m$. And if there exists a nonconstant holomorphic function of linear growth, then the soliton must split off a factor $\mathbb{C}$. 
   \end{itemize}
\end{thm}
Note that by \cite{CM2021}, eigenfunctions of the $f$-Laplacian w.r.t. an eigenvalue $\lambda$ has polynomial growth of at most order $2\lambda$. Theorem \ref{thm: main theorem 1} reveals the relation between polynomial growth holomorphic functions and the spectral data of the $f$-Laplacian, it gives a reproof of the finiteness of $\dim_\mathbb{C} \mathcal{O}_d$, and improves the result of Munteanu-Wang \cite{MW2014} in the case of linear growth holomorphic functions. However, the estimate in Theorem \ref{thm: main theorem 1} (a) is not sharp when $d> 1$.

For linear growth harmonic functions, Li and Tam proved that $dim(\mathcal{H}_1(M))\leq n+1$ on Ricci-nonnegative manifolds, when the manifold is K\"ahler and with nonnegative bisectional curvature, they also showed $M=\mathbb{C}^n$ when the equality is achieved. Rigidity of the real case with nonnegative Ricci curvature was proved by J. Cheeger, T. H. Colding and W. P. Minicozzi {\cite{CCM95}}, who further showed that if $dim(\mathcal{H}_1(M^n))=k+1$ for some integer $k\leq n-1$, then $M_{\infty}\cong \underline{M}_{\infty} \times \mathbb{R}^k$, where $M_{\infty}$ is the tangent cone at infinity of $M$.

Once we know the dimension of $ \mathcal{O}_d$ is finite, we would like to know how $dim\mathcal{O}_d$ depends on the growth rate $d$. Actually, the method of Munteanu and Wang \cite{MW2014} yields an upper bound for $dim_\mathbb{C} \mathcal{O}_d$ as an exponential function of $d$, they further speculated that the upper bound should be a polynomial function of $d$. Our next theorem gives a partial confirmation with some additional curvature assumptions. 
\begin{thm}\label{thm: effective dimension estimate for holomorphic functions}
Let $(M, g, f)$ be a gradient K\"ahler Ricci shrinker with bounded curvature, $\dim_{\mathbb{C}}M = m$, suppose there is a function $S_0(r)$ and a constant $\sigma > 0$ such that 
\[
|S - S_0(b)| \leq b^{-\sigma}
\]
when $b= 2\sqrt{f}$ is large enough, then for all $d \geq 1$,
\[
\dim \mathcal{O}_d \leq C d^{2m-1},
\]
where $C$ is a constant depending on the geometry.
\end{thm}
Note that on a gradient Ricci shrinker, the function $b = 2\sqrt{f}$ is asymptotic to the distance function \cite{CZ2010}, so the assumption in Theorem \ref{thm: effective dimension estimate for holomorphic functions} requires the scalar curvature to be asymptotically radial. For example, in the case when the shrinking soliton is asymptotically conical, we can take $S_0\equiv 0$ and $\sigma = 2$. In dimension four, gradient Ricci shrinkers with scalar curvature decaying to $0$ at infinity must be asymptotically conical by the work of Munteanu and Wang \cite{MW2019}. Theorem \ref{thm: effective dimension estimate for holomorphic functions} can be viewed as a positive evidence that the same phenomenon may still be true in higher dimension. 

Recall that in previous works, an estimate in the form in Theorem \ref{thm: effective dimension estimate for holomorphic functions} usually requires some uniform functional inequality, such as a Neumann Poincare inequality or a Sobolev inequality with uniform constants for all radius; such tools are not available in the current situation. The key ingredient in our proof of Theorem \ref{thm: effective dimension estimate for holomorphic functions} is {\color{black} an almost} sharp estimate for the frequency of holomorphic functions, which may be of independent interest, see Theorem \ref{thm: frequency upper bound by growth upper bound}.  

The frequency function was first defined by Almgren for harmonic functions on $\mathbb{R}^n$ \cite{Al77}, monotonicity of the frequency shows log-convexity of the $L^2$ average of harmonic functions on spheres. Colding and Minicozzi generalized it to harmonic functions on Ricci-nonnegative manifolds \cite{CM1997}, they also studied the frequency for eigenfunctions of the $f$-Laplacian on gradient Ricci shrinkers \cite{CM2021}. Mai and Ou studied a frequency function for harmonic functions on gradient shrinkers with constant scalar curvature \cite{MO22}. In this article, our $L^2$-integrals for holomorphic functions are weighted by the gradient of $2\sqrt{f}$, following \cite{CM2021}.

We also obtain estimates for holomorphic $(p,0)$-forms on gradient K\"ahler Ricci shrinkers with bounded Ricci curvature. Let $\mathcal{O}_\mu(\Lambda^{p,0})$ denote the space of holomorphic $(p,0)$-forms with polynomial growth of at most order $\mu$. We show that $\dim_{\mathbb{C}} \mathcal{O}_\mu(\Lambda^{p,0})$ is finite, and it can be bounded by the dimension of the space of holomorphic functions with certain growth rate. 
\begin{thm}\label{thm: estimate for the dimension of holomorphic forms}
  Let $(M, g, f)$ be a gradient K\"ahler Ricci shrinker with \textcolor{black}{bounded curvature}, then 
  \begin{itemize}
    \item[(a)] the $f$-Hodge Laplacian $\Delta_f^d$ has discrete spectrum, and we have \[
           \dim_\mathbb{C} \mathcal{O}_\mu (\Lambda^{p,0}) \leq \sum_{i = 0}^s mult(\tilde{\lambda}_i),
           \]  
           where $\Lambda = \sup_M|Ric| + \frac{1}{2}$, and $0 \leq \tilde{\lambda}_0 < \tilde{\lambda}_1 < ... < \tilde{\lambda}_s \leq \frac{\mu}{2} + p \Lambda$ are the distinct eigenvalues of $\Delta_f^d$ on $(p,0)$-forms no bigger than $\frac{\mu}{2} + p \Lambda$.
    \item[(b)] there is a constant $C$ depending on $p$, $sup_M |Rm|$ and the dimension $n = 2m$, such that
           \[
           \dim  \mathcal{O_\mu}(\Lambda^{p,0})  \leq \dim \mathcal{O}_{\mu + p} + e^{C(1+\mu)}. 
           \]
  \end{itemize}
\end{thm}
Holomorphic $(p,0)$-forms which are $L^2$ have been studied by Munteanu and Wang in \cite{MW2014} and \cite{MW2015} under different assumptions. Our proof of Theorem \ref{thm: estimate for the dimension of holomorphic forms}(a) is by the same method as for Theorem \ref{thm: main theorem 1}(a), the main difference is that the curvature will affect the estimate for forms. The proof of (b) combines (a) and some ideas from \cite{MW2014}.

The organization of this article is as the following. In section $1$, We study the relation between holomorphic functions and the eigenfunctions of the $f$-Laplacian on a gradient K\"ahler Ricci shrinker, and prove Theorem \ref{thm: main theorem 1}. In section $2$, we  obtain {\color{black} an almost} sharp estimate for the frequency of a holomorphic function with polynomial growth, using which we prove Theorem \ref{thm: effective dimension estimate for holomorphic functions}. Holomorphic forms are studied in section $3$, where Theorem \ref{thm: estimate for the dimension of holomorphic forms} is proved; as a separate issue we also prove a Liouville type theorem for eigen-$1$-forms on gradient shrinking Ricci solitons. 

{\bf Acknowledgement:} We would like to thank Professor Huai-Dong Cao, Professor Yashan Zhang for helpful conversations, and Professor Jiaping Wang, Ovidiu Munteanu and Lihan Wang for their interest in this work. The first author is partially supported by the National Natural Science fundation of China No.12141101 and No.11971400. The second author is partially supported by NSFC grant No.12401074, 12371061, 12371081 and Fundamental Research Funds for the Central Universities No. 20720250006.

\section{Ancient caloric functions and holomorphic functions}

Let $(M, g, f)$ be a gradient shrinking Ricci soliton, normalized so that
\[
Ric + \n\n f = \frac{1}{2}g.
\]
There is a canonical ancient solution of the Ricci flow associated to this gradient shrinking soliton. For $t \in (-\infty, 0)$, let $\tau = -t$, let $\Phi_t$ be a $1$-parameter family of diffeomorphisms generated by
\[
\partial_t \Phi_t(x) = \frac{1}{\tau} \n f \circ \Phi_t(x), \quad t \in (-\infty, 0), \quad with \quad \Phi_{-1} = id.
\]
Define
\[
g(t) = \tau(t) \Phi_t^* g, 
\]
then $g(t), t\in (-\infty, 0)$ is an ancient solution of the Ricci flow. 

Before going to holomorphic functions, let's first study ancient caloric functions along the Ricci flow $g(t)$, i.e. solutions $u(x, t)$ of 
\begin{equation}\label{eqn: heat equation along ancient solution}
\partial_t u(x, t) = \Delta_{g(t)} u(x, t) \quad on \quad M \times (-\infty, 0).
\end{equation}
\begin{lem}\label{lem: transform to f-caloric function}
Suppose $u(x, t)$ is an ancient caloric function along $g(t)$, 
Then the function $\hat{u}(x,s) : = u(\Phi_{-e^{-s}}^{-1} (x), -e^{-s})$ is a solution of the $f$- heat equation:
\begin{equation}\label{eqn: f-heat equation}
\pd_s \hat{u} = \Delta_f \hat{u}(x,s), \quad (x, s) \in M \times (-\infty, \infty).
\end{equation}
The map sending $u(x, t)$ to $\hat{u}(x,s)$ is a bijection between the space of ancient solutions of the heat equation along $(M,g(t), t< 0)$, and the space of eternal solutions of the the $f$-heat equation on $(M,g)$.
\end{lem}
\begin{proof}
Define
\[
\tilde{u}(x, t) = u(\Phi_t^{-1}(x), t), 
\]
then
\[
\partial_t \tilde{u} = \left( \Delta_{\tau \Phi_t^* g} u \right) (\Phi_t^{-1}(x), t) - \frac{1}{\tau} \langle \n f, \n u\rangle (\Phi_t^{-1}(x), t) = \frac{1}{\tau} \left( \Delta_g \tilde{u} - \langle \n f, \n \tilde{u}\rangle_g  \right),
\]
let $s = -\ln (-t) \in (-\infty, \infty)$, and denote $\hat{u}(x, s) = \tilde{u} (x, - e^{-s})$, then we have
\[
\partial_s \hat{u} = \Delta_g \hat{u} - \langle \n f, \n \hat{u} \rangle = : \Delta_f \hat{u} \quad on \quad M \times (-\infty, \infty),
\]
where the covariant derivative and Laplacian are computed w.r.t. the fixed soliton metric $g$, which is independent of $s$. 

Conversely, given a solution of (\ref{eqn: f-heat equation}) on a gradient shrinking Ricci soliton, we can reverse the above construction to get an ancient caloric function along the canonical ancient Ricci flow associated to it. Hence it is easy to see that solution of (\ref{eqn: heat equation along ancient solution}) and (\ref{eqn: f-heat equation}) are in one to one correspondence.  

\end{proof}

Denote the weighted measure $d\mu_f = e^{-f} dg$, $\Delta_f$ is self-adjoint with respect to the $L^2$ inner product defined with this measure. By \cite{CZ2016} and \cite{HN2014}, on a gradient shrinking Ricci soliton, the Sobolev space $W^{1,2}(M, \mu_f)$ is compactly embedded in $L^2(M, \mu_f)$, hence the operator $\Delta_f$ has discrete spectrum $\lambda_0 = 0 < \lambda_1 < \lambda_2 \leq \lambda_2 \leq ...$, and each eigenvalue has finite multiplicity. Let $\psi_0 = const, \psi_1, \psi_2, ...$ be the corresponding eigenfunctions, which are normalized to form an orthonormal basis of $L^2(M, \mu_f)$.

We claim that each solution $\hat{u}$ of (\ref{eqn: f-heat equation}) with initial value (at time $t = -1$) in $L^2(M, \mu_f)$ can be uniquely written as 
\begin{equation}\label{series solution of f-heat equation}
\hat{u}(x, s) = \sum_{i=0}^\infty a_i e^{-\lambda_i s} \psi_i(x),
\end{equation}
for a unique sequence of coefficients $a_i, i = 0,1,2...$. Indeed, we can find coefficients $a_i, i =0,1,2,...$ so that 
\[ \hat{u}(\cdot, -1) = \sum a_i \phi_i,\]
then (\ref{series solution of f-heat equation}) is clearly a solution of (\ref{eqn: f-heat equation}).
Then the claim follows from the uniqueness of $L^2(M, \mu_f)$ solutions of  the Cauchy problem of the $\Delta_f$-heat equation. Thus we get the following:
\begin{prop}
Let $(M, g(t))_{t\in (-\infty, 0)}$ be the canonical ancient solution associated to a gradient shrinking Ricci soliton $(M, g, f)$, let $u(x, t)$ be an ancient caloric function along $g(t)$, satisfying
\[
\int_M u(x, t)^2  (4\pi\tau)^{-n/2} e^{- f(x, t)} dg(t) \leq C \tau^{2\Lambda}, \quad {\color{black} \tau > 1,}
\]
where $f(x, t) = f(\Phi_t(x))$ and $\Lambda > 0$ is a constant. Then
\[
u(x, t) = \sum_{\lambda_i \leq \Lambda} a_i (-t)^{\lambda_i} \psi_i(x),
\]
for some coefficients $a_i$.
\end{prop}
\begin{proof}
Define $\hat{u}(x, s)$ as in Lemma \ref{lem: transform to f-caloric function} above, where $t = - e^{-s}$. By a change of variable $y = \Phi_t^{-1}(x)$, we can observe that the growth assumption on $u$ implies that
\[
\int_M \hat{u}(x,s)^2 d\mu_f = \int_M u(y, t)^2 e^{- f(\Phi_t(y))} dg(\Phi_t(y)) = \int_M u(y, t)^2 \tau^{-n/2} e^{-f(y, t)}dg(t) \leq C e^{-2\Lambda s},
\]
for all {\color{black} $s\in (-\infty, 0)$}, hence 
\[
\hat{u}(x, s) = \sum_{\lambda_i \leq \Lambda} a_i e^{-\lambda_i s} \psi_i(x).
\]
\end{proof}
As a corollary, we give a new proof of the finite dimensionality of polynomial growth holomorphic functions on gradient K\"ahler Ricci shrinkers, which was first proved by Munteanu and Wang \cite{MW2014}. 
Let $\mathcal{O}_d$ denote the space of holomorphic function with polynomial growth of degree at most $d$, i.e.
\[
\mathcal{O}_d = \{ \bar{\pd} u = 0 : |u(x)| \leq C (1+r(x))^d \quad \text{for some constant}\quad  C \},
\]
where $r(x)$ is the geodesic distance from some fixed point. 
\begin{cor}\label{cor: dimension of holomorphic function is controlled by counting eigenvalues}
The dimension of holomorphic functions with polynomial growth on a gradient shrinking K\"ahler Ricci soliton is finite. More precisely,
\[
\dim_{\mathbb{C}} \mathcal{O}_d \leq 1 + \frac{1}{2}\# \left\{ \frac{1}{2} \leq \lambda(\Delta_f) \leq \frac{d}{2} \right\}
\]
counting multiplicity. In particular, sub-linear holomorphic functions must be constant. 
\end{cor}
\begin{proof}
Let $u$ be a holomorphic function with polynomial growth of degree at most $d$, then we can find a constant $C$ such that
\[
|u(x)| \leq C (1+ r(x))^d,
\]
where $r(x) = r_p(x)$ is the distance to a fixed point $p$, WLOG we can choose $p$ to be a minimal point of the potential function $f$, hence a fixed point of $\Phi_t$. 

Let $u = v + \sqrt{-1}w$ where $v$ and $w$ are real valued function. We first look at the real part $v$. Note that $v$ is harmonic on each time slice of the associated ancient solution, and is independent of time, hence an ancient caloric function along the Ricci flow. 

Observe that 
\[
\partial_\tau r(\Phi_t^{-1}(x) ) =  \frac{1}{\tau}\langle \n r, \n f \rangle (\Phi_t^{-1}(x)) \leq \frac{1}{2\tau} (r( \Phi_t^{-1}(x) ) + c(n)) ,
\]
this differential inequality implies that {\color{black} for $t \leq -1$},
\begin{equation}\label{eqn: control of distance}
  r(\Phi_t^{-1}(x) )+ c(n) \leq \sqrt{\tau}  (r(x) + c(n)  ).  
\end{equation}
Define $\hat{v}$ as in Lemma \ref{lem: transform to f-caloric function}, then we have 
\[
|\hat{v}(x,s)| = |v(\Phi_t^{-1}(x) )|\leq C (1 + r(\Phi^{-1}(x)))^d  \leq C \tau^{\frac{d}{2}} (r(x) + c(n)  )^d,
\]
where $\tau =-t= e^{-s}$, {\color{black} $s \in (-\infty, 0]$}. By this growth rate and (\ref{series solution of f-heat equation}), $\hat{v}$ can be written as
\[
\hat{v}(x, s) = \sum_{\lambda_i \leq d/2} a_i e^{-\lambda_i s} \psi_i(x).
\]
Since the eigenvalues of $\Delta_f$ has finite multiplicity, we see that the space of holomorphic functions with polynomial growth of degree at most $d$ must have dimension $\leq $ the number of eigenvalues of $\Delta_f$ no larger than $d/2$ (counting multiplicity). To make this more precise, we need to consider the dimension of the spaces spanned by both real and imaginary parts of holomorphic functions. 

Let $u_1, u_2, ..., u_m$ be a set of $\mathbb{C}$-linearly independent holomorphic functions with polynomial growth of degree at most $d$, and suppose they are orthogonal to constant functions w.r.t. $L^2(M, \mu_f)$ inner product. Let
\[u_i = v_i + \sqrt{-1}w_i, \quad i = 1,2,...,m\]
where $v_i, w_i$ are real valued functions. 

The set of functions $\{v_1, v_2,..., v_m, w_1, w_2, ..., w_m\}$ are $\mathbb{R}$-linearly independent. Otherwise, suppose there is a set of nontrivial real coefficients $k_1,k_2,..., k_m, l_1, l_2, ..., l_m$, such that 
\[
\sum k_i v_i +\sum l_i w_i = 0,
\]
then the following nontrivial $\mathbb{C}$-linear combination of $u_1,...,u_m$,
\[
\sum k_i u_i - \sqrt{-1} \sum l_i u_i
\]
is a real valued holomorphic function and is orthogonal to constants, hence $0$. 

By the discussion above, we have $2m \leq $ the number of eigenvalues of $\Delta_f$ no bigger than $d/2$, counting multiplicity.

For the last claim, note that the first nonzero eigenvalue $\lambda_1(\Delta_f) \geq \frac{1}{2}$ (\cite{HN2014}), and it is easy to show that $L^2(dv_f)$-integrable $f$-harmonic functions are constant. 
\end{proof}

\begin{cor}\label{cor: decomposition of a holomorphic function}
 Let $\tilde{\lambda}_0, \tilde{\lambda}_1, ..., \tilde{\lambda}_s$ be the distinct eigenvalues of $\Delta_f$ no bigger than $\frac{d}{2}$. For each $u \in \mathcal{O}_d$ on a gradient K\"ahler Ricci shrinker, there are holomoprhic functions $\tilde{u}_i \in \mathcal{O}_{2\tilde{\lambda}_i}$, $i = 1,2,...,s$, such that
\[
u  = \tilde{u}_0+ \tilde{u}_1+  ... + \tilde{u}_s.
\]
Moreover, $\tilde{u}_i$ is an eigenfunction of $\Delta_f$ (hence of $\mathcal{L}_{\n f }$) with respect to the eigenvalue $\tilde{\lambda}_i$.
\end{cor}
\begin{proof}
In the proof of Corollary \ref{cor: dimension of holomorphic function is controlled by counting eigenvalues}, we have obtained a finite expansion for $u \in \mathcal{O}_d$ in terms of eigenfunctions of the $f$-Laplacian. To be precise, let $0= \tilde{\lambda}_0 < \tilde{\lambda}_1 < \tilde{\lambda}_1 < ...$ be the distinct eigenvalues of $\Delta_f$, let $\psi_{i_{1}}, \psi_{i_2}, ..., \psi_{i_{j_i}}$ be a maximal orthomormal set of eigenfunctions w.r.t. $\tilde{\lambda}_i$, then  
\[
u  = \sum_{\lambda_i \leq d/2} \sum_{j = 1}^{j_i} a_{i_j} \psi_{i_j},
\]
where $a_{i_j}$ are complex numbers. 

On a gradient K\"ahler Ricci soliton we have $f_{ij} = f_{\ib \jb} = 0$, hence $\mathcal{L}_{\n f} u = \langle  \n f, \n u \rangle$ is also a holomorphic function. 
Let $u_0 = u$, and define inductively $u_{k+1} = \tilde{\lambda}_s^{-1} \mathcal{L}_{\n f} u_k$, where $\tilde{\lambda}_s$ is the largest eigenvalue less than or equal to $\frac{d}{2}$. Then $u_k, k = 0,1,2,...$ is a sequence of holomorphic functions, hence $\Delta u_k = 0$, therefore we can show by induction that
\[
u_{k} =\tilde{\lambda}_s^{-1} \mathcal{L}_{\n f} u_{k-1} = \tilde{\lambda}_s^{-1} ( - \Delta_f )u_{k-1} = \sum_{ \tilde{\lambda}_i < \tilde{\lambda}_s} \sum_{j = 1}^{j_i} \left(  \frac{\tilde{\lambda}_i}{\tilde{\lambda}_s} \right)^k a_{i_j} \psi_{i_j} +  \sum_{j = 1}^{j_s} a_{s_j} \psi_{s_j}.
\]
Now it is easy to show that $u_k$ converges smoothly on any compact domain to the function
\[
\tilde{u}_s : = \sum_{j = 1}^{j_s} a_{s_j} \psi_{s_j},
\]
thus $\tilde{u}_s$ is holomorphic. Then we can apply the above argument to $u - \tilde{u}_s$ and the proof of the decomposition is finished by induction. 

By \cite{CM2021}, each $\psi_{i_j}$ has polynomial growth with order $\leq 2 \tilde{\lambda}_i$, hence $\tilde{u}_i \in \mathcal{O}_{2\Tilde{\lambda}_i}$.
\end{proof}

It is known that on the Gaussian soliton $(\mathbb{R}^n, g_E, f=\frac{|x|^2}{4})$, the first nonzero eigenvalue is $\lambda_1 = \frac{1}{2}$ with multiplicity $n$. In \cite{CZ2016}, it was proved that on a complete smooth metric measure space $(M,g, e^{-f} dg)$ with $Ric_f \geq \frac{1}{2}g$, the first nonzero eigenvalue $\lambda_1 \geq \frac{1}{2}$ with multiplicity $\leq n$, and if $\lambda_1 = \frac{1}{2}$ with multiplicity $k$, then the space splits off a factor $\mathbb{R}^k$. Hence our result gives a sharp dimension estimate for linear growth holomorphic functions. 
\begin{cor}\label{cor: linear growth holomorphic functions}
On a gradient shrinking K\"ahler Ricci soliton of complex dimension $m$, we have
\[
\dim_{\mathbb{C}}\mathcal{O}_1 \leq m+1,
\]
which achieves equality only on the Gaussian soliton on $\mathbb{C}^m$. And if there exists a nonconstant holomorphic function of linear growth, then the soliton must split off a factor $\mathbb{C}$. 
\end{cor}
\begin{proof}
By \cite{CZ2016}, the multiplicity of $\lambda_1(\Delta_f)$ is at most $2m$, hence the dimension estimate follows corollary \ref{cor: dimension of holomorphic function is controlled by counting eigenvalues}.

To see the splitting, let $u = v + \sqrt{-1}w$ be a holomorphic function with linear growth, $v$ and $w$ being its real and imaginary part.  By Cor \ref{cor: dimension of holomorphic function is controlled by counting eigenvalues} and \cite{CZ2016}, both $v$ and $w$ are eigenfunctions of $\lambda_1(\Delta_f) = \frac{1}{2}$, hence the manifold splits as a Riemannian product $N \times \mathbb{R}^2$, with $v$ and $w$ being the coordinate functions on the $\mathbb{R}^2$ factor.  Since $u$ is holomorphic, $J(\n u) = - \sqrt{-1} \n u$ implies that $J(\n v) = J(\n w)$ and $J(\n w) = - J(\n v)$, hence the complex structure splits and $u$ gives the complex coordinate on the factor $\mathbb{C}$.
\end{proof} 
 
We remark that our dimension estimate is not sharp for polynomial growth rate $d > 1$, for example, on $\mathbb{R}^2$ as a Gaussian soliton, the eigenvalue $1$ has multiplicity $3$ since $x^2-1, y^2-2$ and $xy$ are eigenfunctions, while $\dim_{\mathbb{C}} \mathcal{O}_1(\mathbb{C} )= 2 < 1 + \frac{3}{2}$. 

\begin{proof}[Proof of main Theorem \ref{thm: main theorem 1}]
    Theorem \ref{thm: main theorem 1} follows from Corollary \ref{cor: dimension of holomorphic function is controlled by counting eigenvalues}, \ref{cor: decomposition of a holomorphic function} and \ref{cor: linear growth holomorphic functions}.
\end{proof}

\section{Frequency and dimension estimates for holomorphic functions}

Let $(M, g, f)$ be a complete gradient shrinking Ricci soliton with dimension $n$. The potential function $f$ is normalized so that
\[S + |\n f|^2 = f.\]
It was proved in \cite{CZ2010} that
\begin{equation}\label{eqn: estimate of the potential function}
\frac{1}{4}(r - c_1)_+^2 \leq f \leq \frac{1}{4}(r + c_2)^2,
\end{equation}
where $c_1,c_2$ are constants depending only on the dimension $n$. It is well-known that the scalar curvature $S> 0$ on a gradient Ricci shrinker. 

Define $b= 2 \sqrt{f}$, then $b$ satisfies the following properties which will be frequently used in the following. 
\[
\n b = \frac{\n f}{\sqrt{f}}, \quad |\n b| = \frac{\sqrt{f-S}}{\sqrt{f}} < 1.
\]
\[
b\Delta b = n-1 - (2-\frac{1}{f}) S = n - |\n b|^2 - 2S.
\]
\[
|\n b|^2 = 1 - \frac{4S}{b^2}.
\]
\[
\n\n b = \frac{\sqrt{f} \n\n f - \frac{1}{2\sqrt{f}} |\n f|^2 }{f}.
\]

For simplicity, we denote 
\[ \| S \|_{r, \infty} = \sup_{\{b = r\}} S, \quad \|S\|_\infty = \sup_M S.\]
Assume $S$ is bounded, then clearly $b$ is regular when $b^2 > 4 \|S\|_\infty$.

Define 
\[ V(r) = Vol(\{ b < r\}) \quad A(r) = V'(r) = Area(\{ b = r\}).\]
By \cite{CZ2010}, there is a constant $C> 0$ such that $V(r)\leq C r^n$. In particular, this together with (\ref{eqn: estimate of the potential function}) implies that the weighted measure $dv_f : = e^{-f} dv$ is finite. We can also derive an estimate for $A(r)$.
\begin{lem}\label{lem: estimate of the area of level sets}
   suppose $\|S\|_\infty < \infty$, then there is a constant $C> 0$ such that $A(r)\leq C r^{n-1}$ for $r \geq 8 \|S\|_\infty$.
\end{lem}
\begin{proof}
    By integrating the equation $S + \Delta f = \frac{n}{2}$ and applying the divergence theorem (see Lemma 3.1 in \cite{CZ2010}), we have
    \[
    n V(r)  - r V'(r) = 2 \int_{b < r} S - 2 \int_{b = r} \frac{S}{|\n f|}.
    \]
    Since the scalar curvature $S$ is positive and bounded, we get
    \[
    n V(r) \geq (r - \frac{4\|S\|_\infty}{r}) A(r),
    \]
    hence the result follows from $V(r) \leq C r^n$.
\end{proof}

In this section we consider holomorphic functions on a K\"ahler gradient shrinking Ricci soliton $(M,g,f)$ with real dimension $n = 2m$ and uniformly bounded curvature. We have $f_{ij}=f_{\Bar{i}\Bar{j}}=0$. This important condition can be rephrased as $\nabla f$ being the real part of a holomorphic vector field, or as $J(\nabla f)$ be a Killing vector field on $M$ (see \cite{Cao96}). However, most of the calculation also works for harmonic functions on a gradient shrinking Ricci soliton, the K\"ahler condition and holomorphicity is only need in (\ref{eqn: integration by parts for the Ricci term}) and Lemma \ref{lem: MW}.

Now let $u$ be a holomorphic function. Define
\[
I(r) = r^{1-n} \int_{b=r} |u|^2 |\n b|.
\]
Note that when $r^2 > 4 \| S \|_\infty$, we have $|\n b| > 0$ on the regular level set $\{ b = r\}$, and $I(r) > 0$, otherwise $u \equiv 0$ by maximum principle and unique continuation. 
 
By the divergence theorem,
\[
\begin{split}
I(r) = & r^{1-n} \int_{b=r} |u|^2 \langle \n b, \vn \rangle \\
= & \int_{b<r} \langle \n |u|^2 , \n b \rangle  b^{1-n} - (n-1) |u|^2 b^{-n} |\n b|^2 + |u|^2 b^{1-n} \Delta b \\
= & \int_{b< r} \langle \n |u|^2 , \n b \rangle  b^{1-n} - (n-1) |u|^2 b^{-n} |\n b|^2 + |u|^2 b^{-n} (n - |\n b|^2 - 2S) \\
= & \int_{b< r}  \langle \n |u|^2 , \n b \rangle  b^{1-n} - n |u|^2 b^{-n} (1 - \frac{4S}{b^2}) + |u|^2 b^{-n} (n  - 2S) \\
= & \int_{b< r} \langle \n |u|^2 , \n b \rangle  b^{1-n} + |u|^2 b^{-n} (\frac{4n}{b^2} - 2) S.
\end{split}
\]

The derivative of $I(r)$ in $r$ can be computed by the co-area formula, 
\begin{equation}\label{eqn: derivative of I}
I'(r) = r^{1-n} \int_{b=r} \langle \n |u|^2, \vn \rangle + r^{-n} \left( \frac{4n}{r^2} - 2\right) \int_{b=r} \frac{S |u|^2}{|\n b|}.
\end{equation}
Note that when $r \geq \sqrt{2n}$, we have
\[
I'(r) \leq r^{1-n} \int_{b=r} \langle \n |u|^2, \vn \rangle = : \frac{2 D(r)}{r},
\]
where we define
\[
D(r) = \frac{1}{2} r^{2-n} \int_{b= r}  \langle \n |u|^2, \vn \rangle  =  r^{2-n} \int_{b< r} |\n u|^2,
\]
the last equation comes from the divergence theorem and the computation
\[
\Delta |u|^2 = 2 g^{i \jb}\pd_i \pd_{\jb} (u \bar{u}) = 2 g^{i \jb}\pd_i u \pd_{\jb} \bar{u} = 2 |\pd u|^2 = 2 |\n u|^2.
\]

We will see later that the second term in the RHS of (\ref{eqn: derivative of I}) has to be critically used to cancel a bad term in $D'(r)$.

\begin{lem}\label{lem: first variation formula kahler case}
Let $u$ be a holomorphic function on $\{b< r\}$, let $\vn$ be the unit outer normal vector of the boundary $\{b= r\}$, where $r$ is a regular value of $b$. For any real vector field $X$, we have
\[
\int_{b< r} |\n u|^2 div( X) -  \langle \n u \otimes \n \ub + \n \ub \otimes \n u , \n X \rangle  = \int_{b=r} |\n u|^2 \langle X, \vn\rangle - \langle \n u, X\rangle \frac{\pd \ub}{\pd \vn} - \langle \n \ub, X\rangle \frac{\pd u}{\pd \vn}.  
\]
\end{lem}
\begin{proof}
We prove this equation using integration by parts. Denote the unit normal vector as 
\[
\vn  = \nu_k g^{k \lb} \frac{\pd}{\pd \bar{z}^l}  + \nu_{\kb} g^{l \kb} \frac{\pd}{\pd {z}^l}.
\]
\[
\begin{split}
\int_{b< r} |\n u|^2 div(X) = & \int_{b< r} - g^{i \jb} \pd_i |\n u|^2 X_{\jb} - g^{ij} \pd_{\jb} |\n u|^2 X_i + \int_{b=r} |\n u|^2 \langle X, \vn \rangle \\
= &  \int_{b< r} -  g^{i \jb} g^{k \lb} u_{ik} \ub_{\lb} X_{\jb} - g^{i \jb} g^{k \lb} u_k \ub_{\jb \lb} X_i + \int_{b=r} |\n u|^2 \langle X, \vn \rangle \\
= & \int_{b< r}   g^{i \jb} g^{k \lb} u_{i} \ub_{\lb} \n _k X_{\jb} + g^{i \jb} g^{k \lb} u_k \ub_{\jb } \n_{\lb} X_i + \int_{b=r} |\n u|^2 \langle X, \vn \rangle \\
& - \int_{b=r} g^{i \jb} g^{k \lb} u_{i} \ub_{\lb}  X_{\jb} \nu_k + g^{i \jb} g^{k \lb} u_k \ub_{\jb }  X_i \nu_{\lb} .
\end{split}
\]
The proof is finished by rewriting the RHS in invariant form. 
\end{proof}

\begin{lem}\label{lem: derivative of D} For a holomorphic function $u$ and a regular value $r$ of $b$,
\[
\begin{split}
D'(r) 
= & 2 r^{2-n} \int_{b= r} \left| \frac{\pd u}{\pd \vn }\right|^2 |\n b|^{-1} +  4r^{-n} \left(  \int_{b=r} \frac{S|\n u|^2}{|\n b|} - 2 S \left| \frac{\pd u}{\pd \vn }\right|^2 |\n b|^{-1} \right) \\
&  +  r^{1-n}\int_{b= r} -  S \frac{\pd |u|^2}{\pd \vn}  + 2Ric( \vn, \n |u|^2). \\
\end{split}
\]
\end{lem}
\begin{proof}
Take $X = \n f$ in Lemma \ref{lem: first variation formula kahler case}, we get
\begin{equation}\label{eqn: first variation formula with X = nabla f}
\begin{split}
\int_{b< r} |\n u| ^2 \Delta f - 2 f_{i \jb}u_{\ib} u_j= & \int_{b=r} |\n u|^2 |\n f| - u_i f_{\ib} \frac{\pd \ub}{\pd \vn} -  \ub_{\ib} f_i  \frac{\pd u}{\pd \vn} \\
= & \frac{r}{2} \int_{b=r} |\n u|^2 |\n b| - 2 \left| \frac{\pd u}{\pd \vn} \right|^2 |\n b|,
\end{split}
\end{equation}
by $|\n b|^2 = 1 - \frac{4S}{b^2}$, the RHS can be written as
\[
RHS = \frac{r}{2} \int_{b= r} \frac{|\n u|^2 }{|\n b|} \left( 1 - \frac{4S}{b^2}\right) - 2\frac{ \left| \frac{\pd u}{\pd \vn} \right|^2}{|\n b|} \left( 1 - \frac{4S}{b^2}\right),
\]
By 
\[
\Delta f = \frac{n}{2} - S, \quad f_{i \jb} = \frac{1}{2} g_{i \jb} - R_{i \jb},
\]
the LHS of (\ref{eqn: first variation formula with X = nabla f}) can be written as 
\[
LHS = \int_{b< r} \frac{n-2}{2} |\n u|^2 - S |\n u|^2 + 2 R_{i\jb} u_j \ub_{\ib}.
\]
Hence we have 
\[
\begin{split}
\int_{b=r} \frac{|\n u|^2}{|\n b|} - \frac{n-2}{r} \int_{b< r} |\n u|^2 = & 2 \int_{b= r} \left| \frac{\pd u}{\pd \vn }\right|^2 |\n b|^{-1} +  \frac{4}{r^2} \int_{b=r} \frac{S|\n u|^2}{|\n b|} - \frac{8}{r^2} \int_{b= r} S \left| \frac{\pd u}{\pd \vn }\right|^2 |\n b|^{-1}\\
&  - \frac{2}{r} \int_{b< r} S |\n u|^2 + \frac{4}{r} \int_{b< r} R_{i\jb} u_j \ub_{\ib}. \\
\end{split}
\]
Using integration by parts and the contracted 2nd Bianchi identity,
\begin{equation}\label{eqn: integration by parts for the Ricci term}
\begin{split}
\int_{b< r} R_{i\jb} u_j \ub_{\ib}= & \frac{1}{2}\int_{b< r} -\n_{\ib} R_{i \jb} u_j \ub -  \n_j R_{i \jb} \ub_{\ib} u  + \frac{1}{2} \int_{b=r}  Ric( \vn, \n u) \ub  + Ric(\vn, \n \ub) u\\
= & \frac{1}{4} \int_{b< r} -  \n_{\jb} S u_j \ub - \n_i S \ub_{\ib} u + \frac{1}{2} \int_{b=r}  Ric( \vn, \n u) \ub  + Ric(\vn, \n \ub) u\\
= &  \frac{1}{2}\int_{b< r}  S |\n u|^2 + \int_{b= r} - \frac{1}{4} S \ub \frac{\pd u}{\pd \vn} - \frac{1}{4} S u \frac{\pd \ub}{\pd \vn}  + \frac{1}{2}Ric( \vn, \n u) \ub  + \frac{1}{2}Ric(\vn, \n \ub) u,
\end{split}
\end{equation}
note that $\bar{\pd} u = 0$ is crucial in the above calculation. 

Thus
\[
\begin{split}
D'(r) = & r^{2-n} \left( \int_{b=r} \frac{|\n u|^2}{|\n b|} - \frac{n-2}{r} \int_{b< r} |\n u|^2 \right) \\
= & 2 r^{2-n} \int_{b= r} \left| \frac{\pd u}{\pd \vn }\right|^2 |\n b|^{-1} +  4r^{-n} \left(  \int_{b=r} \frac{S|\n u|^2}{|\n b|} - 2 S \left| \frac{\pd u}{\pd \vn }\right|^2 |\n b|^{-1} \right) \\
&  +  r^{1-n}\int_{b= r} -  S \ub \frac{\pd u}{\pd \vn} - S u \frac{\pd \ub}{\pd \vn}  + 2Ric( \vn, \n u) \ub  + 2Ric(\vn, \n \ub) u. \\
\end{split}
\]
\end{proof}

\begin{lem}\label{lem: iterating the first variation formula kahler case}
There is a constant $C$ depending on $\|S\|_\infty, \|\n S\|_\infty$ and $\|\n\n b\|_\infty$, such that when $r^2> 4 \|S\|_\infty$, we have
\[
\left| \int_{b=r} \frac{|\n u|^2}{|\n b|} - 2 \int_{b= r}  \left| \frac{\pd u}{\pd \vn }\right|^2 |\n b|^{-1} \right| \leq  C \int_{b< r} |\n u|^2,
\]
\[
\left| \int_{b=r} \frac{S|\n u|^2}{|\n b|} - 2 \int_{b= r} S \left| \frac{\pd u}{\pd \vn }\right|^2 |\n b|^{-1} \right| \leq  C \int_{b< r} |\n u|^2,
\]
for any holomorphic function $u$ with polynomial growth. 
\end{lem}
\begin{proof}
For any integer $j \geq 0$, define
\[
K_j(r) = \int_{b=r} S^j \left( |\n u|^2 - 2 \left| \frac{\pd u}{\pd \vn }\right|^2 \right) |\n b|^{-1}.
\]
By the identity $1 = |\n b|^2 + \frac{4S}{b^2}$, we can write 
\[
\begin{split}
K_j(r) = \int _{b=r}  S^j\left( |\n u|^2 - 2 \left| \frac{\pd u}{\pd \vn }\right|^2 \right) |\n b| + \frac{4}{r^2}K_{j+1}(r).
\end{split}
\]
By applying Lemma \ref{lem: first variation formula kahler case} with $X = S^j \n b$, we have
\[
\begin{split}
& K_j(r) - \frac{4}{r^2} K_{j+1}(r) \\
= &  \int_{b< r} |\n u|^2 div\left( S^j \n b\right) -  \langle \n u \otimes \n \ub + \n \ub \otimes \n u, \n (S^j \n b) \rangle \\
= & \int_{b< r} |\n u|^2 S^{ j-1 } \left( j \langle \n S, \n b\rangle  + S \Delta b -   \langle \n u \otimes \n \ub + \n \ub \otimes \n u  , m \n S \otimes \n b + S \n\n b \rangle \right)\\
\geq & - j \|S\|_\infty^{j-1} C_1 \int_{b< r} |\n u|^2,
\end{split}
\]
when $j \geq 1$ and 
\[
K_0(r) - \frac{4}{r^2} K_1(r) \geq - C_1 \int_{b < r} |\n u|^2,
\]
where $C_1$ depends on $\|S\|_\infty, \|\n S\|_\infty$ and $\|\n\n b\|_\infty$ and is independent of $m$. 

By iterating the above inequality, we get
\begin{equation}\label{eqn: iterating K_m}
K_1(r) \geq \left( \frac{4}{r^2} \right)^j K_{j+1}(r)- \sum_{k=1}^j k \left(\frac{4\|S\|_\infty}{r^2}\right)^{k-1}  C_1 \int_{b< r} |\n u|^2.
\end{equation}
Note that when $r^2 > 4 \|S\|_\infty $, we have
\[
\begin{split}
\left( \frac{4}{r^2} \right)^j K_{j+1}(r) \leq & \left( \frac{4\|S\|_\infty}{r^2} \right)^j \|S\|_\infty \int_{b=r} \left( |\n u|^2 - 2 \left| \frac{\pd u}{\pd \vn }\right|^2 \right) |\n b|^{-1} \\
& \to 0
\end{split}
\]
as $m \to \infty$,  and the series 
\[
\sum_{k=1}^\infty k \left(\frac{4\|S\|_\infty}{r^2}\right)^{k-1}
\]
is convergent. Taking $j \to \infty$ in (\ref{eqn: iterating K_m}) yields 
\[
K_1(r) \geq - C \int_{b< r} |\n u|^2.
\]
$K_0(r)$ can be estimated similarly. And the other side of the inequality can be proved by the same argument.   
\end{proof}

The following Lemma is due to O. Munteanu and J. Wang \cite{MW2014}, we make a slight improvement by reducing the constant $\mu$ to a square of $d$, where it was an exponential function in \cite{MW2014}, this turns out important for our argument. 
\begin{lem}[Munteanu-Wang]\label{lem: MW}
Suppose $u$ is a holomorphic function on $\{b < t\}$, $|u| \leq C(1+r)^d$ for $d\geq 1$, $V(r) \leq C r^p$, then 
\[
\int_{b= r} \left|\frac{\pd u}{\pd \vn} \right|^2 |\n b| \leq \frac{4\mu}{r^2} \int_{b=r} |u|^2 |\n b|^{-1},
\]
for any regular value $c_2\leq r < t$, where $\mu = e^{2p+6}d^2$.
\end{lem}
\begin{proof}
For completeness, we include a proof here following the arguments in \cite{MW2014}.

WLOG we assume that $u$ is a nonconstant holomorphic function. Let $u_0 = u$ and define inductively $u_{k +1} = \langle \n u_k, \n f\rangle $, $k = 0,1,2,...$, then each $u_k$ is a holomorphic function, hence harmonic.
Let
\[
\rho(r) = \frac{\int_{b = r} |u_{1}|^2 |\n b|^{-1}  }{\int_{b = r} |u_{0}|^2 |\n b|^{-1} },
\]
note that our goal is to show $\rho(r)\leq \mu$.

By Green's formula,
\[
\begin{split}
\int_{b = r} u_k \bar{ u}_{k+2}  |\n f|^{-1} - |u_{k+1}|^2 |\n f|^{-1} = & \int_{b = r} u_k  \langle \n \bar{ u}_{k+1}, \vn \rangle - \langle \n u_{k}, \vn \rangle \bar{u}_{k+1} \\
= & \int_{b {\color{black} <} r} u_k \Delta \bar{u}_{k+1} - \bar{u}_{k+1} \Delta u_k = 0,
\end{split}
\]
hence 
\[
\int_{b = r} |u_{k+1}|^2 |\n b|^{-1} = \int_{b = r} u_k \bar{u}_{k+2} |\n b|^{-1}.
\]
By Cauchy-Schwarz inequality, and note that $u_k$ can not vanish on $\{b = r\}$ by the maximum principle, we have
\[
\frac{\int_{b = r} |u_{k+1}|^2 |\n b|^{-1}  }{\int_{b = r} |u_{k}|^2 |\n b|^{-1} } \leq \frac{\int_{b = r} |u_{k+2}|^2 |\n b|^{-1}  }{\int_{b = r} |u_{k+1}|^2 |\n b|^{-1} },
\]
thus
\[
\int_{b = r} |u_{k}|^2 |\n b|^{-1}  \geq  \int_{b = r} |u_{0}|^2 |\n b|^{-1} \prod_{i = 0}^{k-1} \frac{\int_{b = r} |u_{i+1}|^2 |\n b|^{-1}  }{\int_{b = r} |u_{i}|^2 |\n b|^{-1} }  \geq \rho^k(r) \int_{b = r} |u_{0}|^2 |\n b|^{-1}.
\]

Integrating the above inequality yields
\begin{equation}\label{eqn: control of rho^k(r)}
\int_{ r - \delta < b < r} |u_k|^2 {\color{black} delete |\n b|^{-1}} \geq \left(\inf_{r- \delta < s < r} \int_{b = s} |u_0|^2 |\n b|^{-1} \right) \int_{r - \delta}^ r \rho^k (s) ds \geq c \int_{r - \delta}^ r \rho^k (s) ds ,
\end{equation}
since $u_0$ cannot vanish on $\{ b = s\}$ by the maximum principle. 

Next we estimate the LHS of (\ref{eqn: control of rho^k(r)}).  Since each $u_k$ is harmonic,  for a cut-off function $\phi = \phi(b)$ supported on $\{b< (1+\frac{1}{d})r\}$ with $\phi = 1$ on $\{b < r\}$ and $|\phi'|\leq \frac{2d}{r}$, we have 
\[
0 = \int u_k \Delta \bar{u}_k \phi^2 = \int - |\n u_k|^2 \phi^2 + 2 \phi u_k \n_{\ib} \bar{u}_k \n_i \phi ,
\]
hence by Cauchy-Schwarz inequality
\[
\int_{b < r} |\n u_k|^2 \leq \frac{16d^2}{r^2} \int_{b < {\color{black} (1+1/d) } r} |u_k|^2 |\n b|^2 \leq \frac{16d^2}{r^2} \int_{b < (1+1/d)r} |u_k|^2,
\] 
{\color{black} (delete) where $B_p(r)$ is the geodesic ball with radius $r$ centered at $p$.} Hence
\[
\int_{b< r} |u_{k+1}|^2 \leq \int_{b< r} |\n f|^2 |\n u_k|^2 \leq \frac{4d^2(c_2 + r)^2}{r^2} \int_{b< (1+1/d)r} |u_k|^2,
\]
iterating the above inequality finitely many times, and apply the assumptions of the Lemma yields
\[
\begin{split}
\int_{b< r} |u_k|^2 \leq 4^kd^{2k}\left(1 + \frac{c_2}{r} \right)^{2k} \int_{ b < (1+1/d)^k r } |u_0|^2 \leq & C(r,d,p) 16^k d^{2k}\left(1+\frac{1}{d} \right)^{k(2d+p)} \\
 \leq & C(r,d,p) \mu^k,
\end{split}
\]
for $r \geq c_2$.
Together with (\ref{eqn: control of rho^k(r)}) this implies that
\[
\int_{r- \delta}^r \left( \frac{\rho(s)}{\mu} \right)^k ds \leq C(r,p,d), \quad \forall k\geq 1,  
\]
where $C(r,p,d)$ is independent of $k$. Since $\rho(s)$ is continuous, we must have $\rho(s) \leq \mu$ for all $r- \delta < s < r$, for otherwise letting $k \to \infty$ yields a contradiction. 
\end{proof}

Define the frequency function $U(r)$ for a given holomorphic function $u$ by 
\[
U(r) = \frac{D(r)}{I(r)},
\]
when $I(r) > 0$.
Now we can use the above lemmas to derive differential inequalities for $D(r)$ and $U(r)$.

\begin{lem}\label{lem: differential inequality for D}
Let $(M, g, f)$ be a gradient K\"ahler Ricci soliton with bounded curvature, suppose there is a function $S_0(r)$ and a constant $\sigma > 0$ such that 
\[
|S - S_0(b)| \leq b^{-\sigma}
\]
when $b= 2\sqrt{f}$ is large enough, then for any $u \in \mathcal{O}_d$, and for any $\delta \in (0, \frac{1}{2})$, we have
\[
\begin{split}
D'(r) \geq & 2 r^{2-n} \int_{b= r} \left| \frac{\pd u}{\pd \vn }\right|^2 |\n b|^{-1} -   \frac{2S_0(r)}{r} D(r) - \frac{C}{r^{2-\delta}} D(r) - \frac{C\sqrt{\mu}}{r^{1+\min\{\sigma, \delta \}}} I(r),
\end{split}
\]
when $r> R_0$, the constant $C$ depends only on curvature bounds, and the constant $R_0$ depends only on the geometry, $\mu$ is the same constant as in Lemma \ref{lem: MW}.
\end{lem}
\begin{proof}
By the assumption, we can take a constant $R_0> 0$ large enough depending on the geometry, such that when $r > R_0$, we have
\[
\begin{split}
\int_{b = r} - S  \frac{\pd |u|^2}{\pd \vn} \geq & - S_0(r) \int_{b = r} \frac{\pd |u|^2}{\pd \vn} - r^{-\sigma} \int_{b = r} \left|\frac{\pd |u|^2}{\pd \vn} \right| \\
\geq & - 2 S_0(r) r^{n-2} D(r) - 2 r^{-\sigma} \left(\int_{b = r} \left| \frac{\pd u}{\pd \vn}\right|^2 |\n b|^{-1} \right)^{1/2} \left( \int_{b = r} |u|^2 |\n b| \right)^{1/2}.
\end{split}
\]
Choose $R_0$ large enough such that $1 - 4\| S\|_\infty/  r^2 > \frac{1}{2}$ when $r > R_0$, then 
\[
|\n b| \leq |\n b|^{-1} \leq 2 |\n b|,
\]
and by Lemma \ref{lem: MW}, we have
\[
r^{1-n}\int_{b = r} - S  \frac{\pd |u|^2}{\pd \vn} \geq - \frac{2 S_0(r) }{r} D(r) - \frac{8\sqrt{\mu }}{r^{1+\sigma}}  I(r) .
\]
For the Ricci curvature term in $D'(r)$ computed in Lemma \ref{lem: derivative of D}, recall that on a gradient shrinking Ricci soliton we have
\[
Ric(\n f) = \frac{1}{2} \n S,
\]
hence
\[
\int_{b = r} 2 Ric(\vn, \n |u|^2)
= \frac{2}{r}\int_{b = r} 2 Ric(\n f, \n |u|^2) |\n b|^{-1}
= \frac{2}{r}\int_{b = r} \langle \n S, \n |u|^2 \rangle |\n b|^{-1}.
\]
Since we have assumed bounded curvature on the soliton, Shi's estimate provides uniform control on the derivatives of the curvature, thus we have
\[
\begin{split}
&\int_{b = r} 2 Ric(\vn, \n |u|^2)\\
\geq & - \frac{2 \|\n S\|_\infty}{r} \left( \int_{b = r} |\n u|^2 |\n b|^{-1} \right)^{1/2} \left( \int_{b = r} |u|^2 |\n b|^{-1} \right)^{1/2}\\
\geq & - \frac{2 \|\n S\|_\infty}{r} \left( 2 \int_{b = r} |\frac{ \pd u}{\pd \vn}|^2 |\n b|^{-1}  + C \int_{b< r} |\n u|^2 \right)^{1/2} \left( \int_{b = r} |u|^2 |\n b|^{-1} \right)^{1/2},\\
\end{split}
\]
where we have used Lemma \ref{lem: iterating the first variation formula kahler case} to get the last inequality, the constant $C$ depends on curvature bounds. By elementary inequalities and the Lemma \ref{lem: MW}, 
\[
\begin{split}
& r^{1-n}\int_{b = r} 2 Ric(\vn, \n |u|^2) \\
\geq & - \frac{2 \|\n S\|_\infty}{r^n} \left( \left(2 \int_{b = r} |\frac{ \pd u}{\pd \vn}|^2 |\n b|^{-1}\right)^{1/2} +   C \left(\int_{b< r} |\n u|^2 \right)^{1/2} \right) \left( \int_{b = r} |u|^2 |\n b|^{-1} \right)^{1/2}\\
\geq & - \frac{2 \|\n S\|_\infty}{r^n} \left( \frac{8\sqrt{\mu}}{r}\int_{b = r} |u|^2 |\n b| + C r^\delta \int_{b< r} |\n u|^2 + \frac{2}{r^\delta} \int_{b = r} |u|^2 |\n b|\right) \\
\geq & -  \frac{C}{r^{2-\delta}} D(r) - \frac{C\sqrt{\mu}}{r^{1+\delta}} I(r),
\end{split}
\]
where we take $\delta \in (0, \frac{1}{2})$, the constant $C$ depends on curvature bounds.

Then the result follows from the above calculation and Lemma \ref{lem: derivative of D}, \ref{lem: iterating the first variation formula kahler case}.
\end{proof}

\begin{lem}\label{lem: differential inequality for U}
Under the same assumption as in Lemma \ref{lem: differential inequality for D},
\[
U'(r) \geq -\frac{C_1}{r^{1+\min \{\sigma, 1 - \delta\}}} U(r) - \frac{C_2\sqrt{\mu}}{r^{1+\min\{ \sigma, \delta\}}},
\]
when $r > R_0$, where $C_1, C_2$ are constants depending on the curvature bound, $R_0$ is the same as in Lemma \ref{lem: differential inequality for D}. 

In particular, if $\sigma \in (0, \frac{1}{2}]$, we can take $\delta = \frac{1}{2}$ and the quantity 
\[
U(r) \exp\left(- \frac{C_1}{\sigma r^{\sigma}} \right) + \int_0^r \frac{C_2 \sqrt{\mu}}{s^{1+\sigma}}\exp\left({- \frac{C_1}{\sigma s^{\sigma}}} \right) ds
\]
is monotone nondecreasing. 
\end{lem}
\begin{proof}
By (\ref{eqn: derivative of I}),
\[
\begin{split}
I'(r) =& r^{1-n} \int_{b=r} \langle \n u^2, \vn \rangle + \frac{1}{r^n} \left( \frac{4n}{r^2} - 2\right) \int_{b = r} S|u|^2 \left( |\n b| +  \frac{4S}{ r^2 |\n b|} \right) \\
\leq & \frac{2}{r} D(r) - \frac{2}{r^n} \int_{b= r} S |u|^2 |\n b| + \frac{C}{r^2} I(r) \\
\leq & \frac{2}{r} D(r) - \frac{2S_0(r)}{r} I(r) + \frac{C}{r^{1+\sigma}} I(r),
\end{split}
\]
when $r > R_0$,  where the constant $C$ depends only on the curvature bound. 

For simplicity, we can assume WLOG that $\sigma \in (0, \frac{1}{2})$, and take $\delta = \sigma$, then 
\[
\begin{split}
I^2(r)U'(r) \geq & 2 r^{2-n} \left( \int_{b = r} \left|\frac{\pd u}{\pd \vn}\right|^2 |\n b|^{-1} \right) I - \frac{2S_0(r)}{r} DI - \frac{C}{r^{2-\delta}} DI - \frac{C\sqrt{\mu}}{r^{1+\min\{\sigma, \delta\}}} I^2 \\
& - \frac{2}{r} D^2 + \frac{2S_0(r)}{r} DI - \frac{C}{r^{1+\sigma}}DI,
\end{split}
\]
by Holder inequality,
\[
2 r^{2-n} \left( \int_{b = r} \left|\frac{\pd u}{\pd \vn}\right|^2 |\n b|^{-1} \right) I \geq \frac{2}{r} D^2,
\]
thus we get
\[
U'(r) \geq -\frac{C_1}{r^{1+\min \{\sigma, 1 - \delta\}}} U - \frac{C_2\sqrt{\mu}}{r^{1+\min\{ \sigma, \delta\}}}.
\]
The last statement comes from integrating the above inequality. 
\end{proof}
For convenience, we define
\[
\eta(r) = \int_0^r \frac{C_2 \sqrt{\mu}}{s^{1+\sigma}}\exp\left({- \frac{C_1}{\sigma s^{\sigma}}} \right) ds.
\]
\begin{thm}\label{thm: frequency upper bound by growth upper bound}
Let $(M, g, f)$ be a gradient K\"ahler Ricci soliton with bounded curvature, suppose there is a function $S_0(r)$ and a constant $\sigma > 0$ such that 
\[
|S - S_0(b)| \leq b^{-\sigma}
\]
when $b= 2\sqrt{f}$ is large enough, then for any $\epsilon > 0$, there exists a constant $R_0$ depending on the geometry and the constant $\epsilon$, such that for any $u \in \mathcal{O}_d$,  we have 
\[
U(r) \leq d {\color{black} + \bar{S}_0} +  \epsilon {\color{black} \sqrt{\mu} },
\]
where $\mu$ is the same as in Lemma \ref{lem: MW} {\color{black}, $\bar{S}_0 = \limsup_{b \to \infty} S_0(b)$. }
\end{thm}


\begin{proof}
By Monotonicity of the quantity in Lemma \ref{lem: differential inequality for U}, we have for any $R>r > R_0$, and for any holomorphic function $u$, the frequency function satisfies
\begin{equation}\label{eqn: monotonicity of U}
U(r) \leq \exp\left(\frac{C_1}{\sigma r^{\sigma}} \right) \left( \exp\left(- \frac{C_1}{\sigma R^{\sigma }} \right)U(R) + \eta(R) - \eta(r) \right).
\end{equation}

Claim: For a fixed holomorphic function $u$, for any $\epsilon > 0$, we can find a sequence of $r_k \to \infty$, such that 
\[
U(r_k) \leq d {\color{black} + \bar{S}_0} + \epsilon.
\]
\textit{Proof of Claim:} Argue by contradiction. If there does not exist such a sequence, then there is an $\epsilon_0> 0$ and $R_1$ which may depend on the function $u$, such that for any $r> 0$, $U(r) > d {\color{black} + \bar{S}_0} + \epsilon_0$. Then by (\ref{eqn: derivative of I}),
\[
r(\ln I)' = 2U - \left( 2- \frac{4n}{r^2} \right) \frac{\|S\|_{r,\infty}}{1- \frac{4\|S\|_{r, \infty}}{r^2}} > 2d + \epsilon_0,
\]
when $r \geq  R_2 > R_1$, where $R_2$ depends on $R_1, n, \epsilon_0$ and $|S|_\infty$, after integration we get
\[
I(r) \geq I(R_2) \left( \frac{r}{ R_2}\right)^{2d+\epsilon_0}.
\]
On the other hand, by the definition of $I(r)$ and Lemma \ref{lem: estimate of the area of level sets}, 
\[
I(r) \leq C (1+r)^{2d},
\]
which yields a contradiction when $r$ is large enough, and the Claim is proved. 

By the Claim, we can take $R = r_k$ and let $k \to \infty$, note that
\[ \eta(R) - \eta(r) = \int_r^R \frac{C_2 \sqrt{\mu}}{s^{1+\sigma}}\exp\left({- \frac{C_1}{\sigma s^{\sigma }}} \right) ds \leq  \frac{C \sqrt{\mu}}{ r^{ \sigma}}, \]
so there exists a $R_0$ depending only on $C_1, C_2, \sigma $ and $\epsilon$, such that when $r > R_0$, the inequality (\ref{eqn: monotonicity of U}) implies
\[
U(r) \leq d {\color{black} + \bar{S}_0}  +  \epsilon \sqrt{\mu}.
\]
In particular, we note that $R_0$ is independent of the function $u$ and the polynomial growth order $d$.
\end{proof}

As a direct corollary, we get the following three circle type estimate and doubling estimate, which will not be used in the rest of this article.
\begin{cor}
Under the same assumption of Theorem \ref{thm: frequency upper bound by growth upper bound}. 
\[
\frac{I(2\rho)}{I(\rho)} \leq \left(\frac{I(4\rho)}{I(2\rho)} \right)^{\exp ( \frac{C_1}{\sigma \rho^\sigma })} \exp \left(  \frac{C_2 \sqrt{\mu}\ln 2 }{\sigma}\left( \frac{1}{(2\rho)^\sigma} - \frac{1}{(4\rho)^\sigma}\right) e^{\frac{C_1}{\sigma \rho^\sigma }} {\color{black} + 8 \bar{S}_0 e^{\frac{C_1}{\sigma \rho^\sigma}} }\right),
\]
and for any $\epsilon > 0$ there is an $R_0$, such that
\[
I(2\rho) \leq 2^{2(d {\color{black} + \bar{S}_0} +  \epsilon \sqrt{\mu}) } I(\rho),
\]
when $\rho> R_0$.
\end{cor}
\begin{proof}
For convenience, we define
\[
\xi(r) = \frac{1}{r} e^{\frac{C_1}{\sigma r^\sigma}}.
\]
By (\ref{eqn: derivative of I}) and (\ref{eqn: monotonicity of U}), we have 
\[
\xi(R) (\ln I)'(r) \leq \xi(r) (\ln I)'(R) + \xi(r) \xi(R) (\eta(R) - \eta(r)) + \frac{4\|S\|_{R, \infty}}{R} \xi(r). 
\]
Integrate on intervals $[\rho, 2 \rho]$ and $[2 \rho, 4 \rho]$ w.r.t. variables $r$ and $R$ respectively, we get
\[
\begin{split}
\left( \int_{2\rho}^{4\rho} \xi(R)dR \right) \ln \left( \frac{I(2\rho)}{I(\rho)} \right) \leq & \left( \int_{\rho}^{2\rho} \xi(r)dr \right) \ln \left( \frac{I(4\rho)}{I(2\rho)} \right) \\
& + \int_{2\rho}^{4\rho} \int_{\rho}^{2\rho} \xi(r)\xi(R) (\eta(R) - \eta(r)) dr dR \\
& + 4 \int_{\rho}^{2\rho} \xi(r)  dr \int_{2\rho}^{4\rho} \frac{1}{R^{1+\sigma}} dR {\color{black} + 4 (\ln 2) \bar{S}_0\int_{\rho}^{2\rho} \xi(r) dr} .
\end{split}
\]
By elementary estimates of the integrals, we get
\[
\ln \left( \frac{I(2\rho)}{I(\rho)} \right) \leq e^{\frac{C_1}{\sigma \rho^\sigma }}  \ln \left( \frac{I(4\rho)}{I(2\rho)} \right) + \frac{C_2 \sqrt{\mu} \ln 2 }{\sigma}\left( \frac{1}{(2\rho)^\sigma} - \frac{1}{(4\rho)^\sigma}\right) e^{\frac{C_1}{\sigma \rho^\sigma }} {\color{black} + 8 \bar{S}_0 e^{\frac{C_1}{\sigma \rho^\sigma}} }, 
\]
which is the three circle type estimate. 

The doubling estimate comes from integrating
\[
(\ln I)' \leq \frac{2U}{r} \leq \frac{2(d {\color{black} + \bar{S}_0} + \epsilon \sqrt{\mu}) }{r}.
\]

\end{proof}

Now we can prove the second main theorem in the introduction. 
\begin{proof}[Proof of Theorem \ref{thm: effective dimension estimate for holomorphic functions}]
By (\ref{eqn: derivative of I}), 
\[
\frac{2 D(r)}{r} \geq I'(r) \geq - \frac{2\|S\|_{r, \infty}}{r} \frac{1}{1- \frac{4\|S\|_{r, \infty}}{r^2} }  I(r),
\]
when $r \geq \sqrt{2n}$. Hence by Theorem \ref{thm: frequency upper bound by growth upper bound}, we can choose $R_0$ large enough for each given $\epsilon > 0$, such that 
\begin{equation}\label{eqn: upper and lower bounds for (ln I)'}
 - \frac{C}{r} \leq ( \ln I)' (r)\leq \frac{2(d {\color{black} + \bar{S}_0} +\epsilon \sqrt{\mu})}{r} 
\end{equation}
where $C$ only depends on the curvature bound, for all $r> R_0$, and for any holomorphic function $u$ of polynomial growth with degree at most $d$. WLOG we also require that 
\[
|\n b|^2 \geq 1-\epsilon,
\]
when $b> R_0$.

Now choose $\lambda = 1 + \frac{2}{d}$, integrating (\ref{eqn: upper and lower bounds for (ln I)'}) yields
\begin{equation}\label{eqn: comparing I(lambda r) and I(r)}
\lambda^{-C} I(r) \leq e^{-(1-1/\lambda^\sigma) \frac{2+\epsilon}{\sigma r^\sigma}} I(r) \leq I(\lambda r) \leq \lambda^{2(d {\color{black} + \bar{S}_0} +\epsilon \sqrt{\mu})} I(r).
\end{equation}
Choose $r_0 = R_0$, let $r_i = \lambda r_{i-1}$ for $i = 1,2,3$, define 
\[
J_0 = 0, \quad J_i = \int_{r_{0} < b < r_i} |u|^2 |\n b|^2 , \quad i = 1,2,3.
\]
Then, by a change of variable $r = \lambda s$,  
\[
J_{i+1} -J_i = \int _{r_{i}}^{r_{i+1}} r^{n-1} I(r) dr = \int_{r_{i-1}}^{r_i} (\lambda s)^{n-1} I(\lambda s) \lambda ds,
\]
by (\ref{eqn: comparing I(lambda r) and I(r)}),
\[
\lambda^ {n-C} (J_i - J_{i-1}) \leq J_{i +1} -J_i\leq \lambda^{n+ 2(d {\color{black} + \bar{S}_0} +\epsilon \sqrt{\mu})} (J_i - J_{i-1}), \quad for \quad i = 1,2.
\]
Hence 
\[
J_1 \leq \lambda^ {C-n} (J_2 - J_1), 
\]
and 
\begin{equation}\label{eqn: control J_3 by J_2 - J_1}
J_3 \leq J_1 + ( 1+ \lambda^{n+ 2(d {\color{black} + \bar{S}_0} +\epsilon \sqrt{\mu})} )  (J_2 - J_1) \leq (1+ \lambda^{n+ 2(d {\color{black} + \bar{S}_0} +\epsilon \sqrt{\mu})}+  \lambda^ {C-n}) (J_2 - J_1).
\end{equation}
Define a Hermitian inner product on $L^2(\{r_1 < b < r_2\})$, for $L^2$ functions $u,v$,  
\[
K(u, v) := \int_{r_1 < b < r_2} u\bar{v} |\n b|^2.
\]
Let $\mathcal {P}$ be a finite dimensional subspace of the space of holomorphic functions of polynomial growth with degree at most $d$, suppose $dim \mathcal{P} = k$. Let 
\[u_1, u_2, ..., u_k\]
be any unitary basis of $\mathcal{P}$ w.r.t. the inner product $K(\cdot, \cdot)$.

For any $x \in \{r_1 < b < r_2\}$, there is a subspace of $\mathcal{P}$ of codimension at most $1$, which vanishes at the point $x$ (see Chapter 7 of \cite{Li2012}). Up to a unitary change of basis, we can assume that $u_1(x) = u_2(x) = ... = u_{k-1}(x) = 0$, hence
\[
\sum_{i = 1}^k u_i^2(x) = u_k^2(x).
\]
Since we have assumed uniformly bounded curvature, and the gradient shrinking soliton is locally noncollapsed (\cite{CN2009}), there is a constant $v_0$ independent of $x$ such that 
\[Vol(B_x(d^{-1})) \geq v_0 d^{-n},\]
where $n$ is the real dimension of the manifold. Hence there is a local mean value inequality on $B_x(d^{-1})$, which tells that
\[
|u_k(x)|^2 \leq C_M d^{n} \int_{B_x(d^{-1})} |u_k|^2,
\]
where $C_M$ depends on $n$, Ricci curvature lower bound and $v_0$. By taking $\epsilon$ small enough, we have $B_x(d^{-1}) \subset \{r_0 < b < r_3\}$ for any $x \in \{r_1 < b < r_2\}$ and for all $d$, thus
\[
\sum_{i = 1}^k |u_i|^2(x) = |u_k|^2(x) \leq C_M d^{n} \int_{B_x(d^{-1})} |u_k|^2 \leq C_M d^{n} \sup_{u \in \mathcal {P}, J_2(u) - J_1(u) = 1} J_3(u).
\]
Note that the LHS is invariant under unitary change of basis, so the above inequality holds for any $x \in \{r_1 < b < r_2\}$. Now integrate the above inequality on $\{r_1 < b < r_2\}$ and apply (\ref{eqn: control J_3 by J_2 - J_1}),  we get
\[
k \leq C_M d^{n} |\{r_1 < b < r_2\}| (1+ \lambda^{n+ 2(d {\color{black} + \bar{S}_0} +\epsilon \sqrt{\mu})}+ \lambda^ {C-n}),
\]
Integrating the estimate in Lemma \ref{lem: estimate of the area of level sets} yields
\[
 |\{r_1 < b < r_2\}| \leq C (\lambda^{2n} - \lambda^{n} ) R_0^{n} \leq C d^{-1} R_0^{n}, 
\]
hence
\[
k \leq C d^{n-1} \left(1 + \left(1+\frac{2}{d} \right)^{2(d {\color{black} + \bar{S}_0} +\epsilon \sqrt{\mu})} \right),
\]
note that in Lemma \ref{lem: MW} we have $\mu \leq C d^2$, hence we can fix a small value for $\epsilon$ and get
\[
k \leq C d^{n-1}.
\]
We have allowed the constant $C$ to change from line to line, while it only depends on the geometry and is independent of $d$. Note that $n = 2m$ where $m$ is the complex dimension.
\end{proof}

\section{Holomorphic forms}

Let $(M, g, f)$ be a gradient K\"ahler Ricci shrinker, normalized such that
\[ 
R_{i \jb} + f_{i \jb} = \frac{1}{2} g_{i \jb}.
\]
We use $\mathcal{O}_\mu$ to denote the linear space of holomorphic function of polynomial growth with order at most $\mu$, and $\mathcal{O}_\mu(\Lambda^{p,q})$ denote the linear space of holomorphic $(p,q)$-forms of polynomial growth with order at most $\mu$.

Let $\Delta^d = d d^* + d^* d$ be the standard Hodge Laplacian. Since $(M, g)$ is K\"ahler, holomorphic forms are naturally $\Delta^d$-harmonic. 

Let $dv_f = e^{-f }dv$ be a weighted measure on $(M,g)$, define $d^*_f$ to be the $L^2$-dual of $d$ with respect to $dv_f$, i.e.
\[
\int \langle d \w ,\eta \rangle dv_f = \int \langle \w, d_f^* \eta \rangle dv_f 
\]
for any compactly supported smooth $(p-1)$-form $\w$ and $p$-form  $\eta$. Define the weighted Hodge Laplacian by $\Delta^d_ f = d d^*_f + d^*_f d$. This operator was studied in \cite{MW2015}, where it was shown that  
\[
\Delta^d_f = \Delta^d + \mathcal{L}_{\n f},
\]
and $\Delta^d_ f $ preserves the type of $(p, q)$-forms if and only if $\n f$ is a real holomorphic vector field, i.e. the $(1, 0)$ part of $\n f$ is a holomorphic vector field, which naturally holds on a gradient K\"ahler Ricci shrinker. 

Let $\tau = -t \in (0, \infty)$, let $\Phi_t$ be the $1$-parameter family of diffeomorphism generated by
\[
\partial_t \Phi_t(x) = \frac{1}{\tau} \n f \circ \Phi_t(x), \quad t \in (-\infty, 0), \quad with \quad \Phi_{-1} = id.
\]
Since $\n f$ is a real holomorphic vector field, $\Phi_t$ is a family of biholomorphism on $M$. 

\begin{lem}
The function
\[
\hat{\w} (x, s) := (\Phi^{-1}_{-e^{-s}})^*\w(x) , \quad (x, s) \in M \times (-\infty, \infty)
\] 
is an ancient solution of the heat type equation
\begin{equation}\label{eqn: heat equation for the f Hodge Laplacian}
\pd_s \hat {\w} = - \Delta_f^d \hat{\w},
\end{equation}
and the map sending $\w$ to $\hat{\w}$ is a linear injection from the space of holomorphic $(p, 0)$ forms to the space of $(p, 0)$-forms satisfying (\ref{eqn: heat equation for the f Hodge Laplacian}) for $s\in (-\infty, \infty)$. 
\end{lem}
\begin{proof}
Let $\w$ be a holomorphic $(p, 0)$ form.
Let
\[
g(t) = \tau(t) \Phi_t^* g
\]
be the canonical ancient solution of K\"ahler Ricci flow associated to $(M, g, f)$.

Let $t = - e^{-s}$ for $s \in (-\infty, \infty)$. Then $\hat{\w} (x, s) = (\Phi^{-1}_{t})^*\w(x) $ is a family of holomorphic $(p,0)$-forms. Let $\Delta^d_{g(t)}$ be the Hodge Laplacian w.r.t. the metric $g(t)$. Since $g(t)$ is K\"ahler,
\[
\Delta^d_{g(t)} \hat{\w} = \Delta^{\bar{\pd}}_{g(t)} \hat{\w} = 0,
\]
hence
\[
\pd_s \hat{\w} = e^{-s} \mathcal{L}_{- \frac{1}{\tau} \n f} \hat{\w} = -\Delta^d_f \hat{\w}.
\]
It is easy to see that the map sending $\w$ to $\hat{\w}$ is linear and injective. 
\end{proof}

\begin{lem}\label{lem: spectrum of f-Hodge Laplacian}
 Let $(M, g, dv_f = e^{-f} dv_g)$ be a smooth metric measure space with positive Bakry-Emery Ricci tensor
 \[ Ric + \n\n f \geq a g,\]
 where $a > 0$, and suppose \textcolor{black}{(M, g) has bounded curvature}. Then the $f$-Hodge Laplacian has discrete eigenvalues with finite multiplicity, whose eigen-forms give a complete orthogonal basis for the space $L^2(dv_f)$-forms of the same type.
\end{lem}

\begin{proof}
Since $Ric + \n\n f \geq a g$, we have the well-known Bakry-Emery log-Sobolev inequality 
\[
\int_M u^2 \ln u^2  e^{ -f} dv\leq \frac{2}{a} \int _M |\n u|^2 e^{-f} dv,
\]
for $\int_M u^2 e^{-f} dv = 1$.

Using this log-Sobolev inequality, one can show that $W^{1,2}(dv_f)$ is compactly embedded in $L^2(dv_f)$ (see the appendix of \cite{CZ2016}), hence the $f$-Laplacian densely defined on $L^2(dv_f)$ functions has discrete spectrum, and its (normalized) eigenfunctions consists of a complete orthonormal basis (\cite{HN2014}). 

\textcolor{black}{
The same argument shows that $W^{1,2}(dv_f; \Lambda^p)$ is compactly embedded $L^2(dv_f; \Lambda^p)$. Suppose $(1+\Delta^d_f) \alpha = \beta$, using the Weitzenb\"ock formula $\Delta^d_f = -\Delta + \n_{\n f} + \mathcal{R}$ (see Lemma 3.1 in \cite{HeOu2025}), where $\Delta$ is the rough Laplacian on $\mathcal{R}$ is the curvature term, integration by parts with a cut-off function $\psi$ shows that 
\[
\int \psi^2 |\alpha|^2 + \psi^2 |\n \alpha|^2 + 2\psi \langle \n \alpha, \n \psi \otimes \alpha \rangle + \psi^2\langle \mathcal{R}(\alpha), \alpha\rangle  dv_f = \int \psi^2 \langle \beta, \alpha\rangle dv_f. 
\]
Then, since the curvature is bounded by assumption, use Caucy-Schwarz inequality, and let $\psi \to 1$ yields that $\|\alpha\|_{W^{1,2}(dv_f)} \leq C \|\beta\|_{L^2(dv_f)}$, where $C$ is a constant independent of $\beta$.
 Consequently $(1+\Delta^d_f)^{-1}$ is a compact operator and this finishes the proof.
}

\end{proof}

\begin{prop}\label{series expansion for solutions of the f heat equation of forms}
Let $(M, g, f)$ be a gradient K\"ahler Ricci shrinker with \textcolor{black}{bounded curvature}. There is a constant $\Lambda = \sup_M|Ric| + \frac{1}{2}$, such that for any  $\w \in \mathcal{O_\mu}(\Lambda^{p,0})$, there exist finitely many coefficients $a_i , i = 1,2,...$ such that 
\[
\hat{\w} = \sum_{\lambda_i \leq \mu/2 + p \Lambda} a_i e^{-\lambda_i s} \theta_i,
\]
where $\lambda_i$, $\theta_i$, $i = 1,2,...$ are the eigenvalues and $(p,0)$-eigenforms of $\Delta^d_f$. 
\end{prop}
\begin{proof}
By \cite{MW2015}, $\n f$ being a real holomorphic vector field guarantees that $\Delta^d_f$ preserves the type of  $(p,q)$-forms. Since \textcolor{black}{the curvature} is bounded, by Lemma \ref{lem: spectrum of f-Hodge Laplacian}, the operator $\Delta^d_f$ applied on $L^2(dv_f)$ sections of $(p, 0)$-forms has discrete nonnegative spectrum $0 \leq \lambda_1 < \lambda_2 <  ...$, and the normalized eigenforms provide a complete orthonormal basis, which we denote as $\theta_1, \theta_2, ...$ 

Since $L^2(dv_f)$ solutions to the heat type equation
\[
\pd_s \w = - \Delta^d_f \w
\]
are uniquely determined by the initial data, we can write
\[
\hat{\w} = \sum_{i =1}^\infty a_i e^{-\lambda_i s} \theta_i,
\]
for a sequence of coefficients $a_i$.

Now suppose $\w$ has polynomial growth of order
\[
|\w(x)| \leq C (1+r(x))^\mu,
\]
in local coordinates
\[
\w = \frac{1}{p!} \w_{i_1 i_2 ... i_p} dz^{i_1} \wedge dz^{i_2} \wedge ... \wedge dz^{i_p},
\]
where $|\w_{i_1 i_2 ... i_p}(x)| \leq C (1+r(x))^\mu$.

For simplicity , denote $\phi = \Phi^{-1}_{t}$, then 
\[
\hat{\w}(x,s) = \frac{1}{p!} \w_{i_1 i_2 ... i_p}(\phi(x)) \frac{\pd \phi^{i_1}(x)}{\pd z^{j_1}} \frac{\pd \phi^{i_2}(x)}{\pd z^{j_2}} ... \frac{\pd \phi^{i_p}(x)}{\pd z^{j_p}} dz^{j_1} \wedge dz^{j_2} \wedge ... \wedge dz^{j_p}|_{\phi(x)}.
\]
Note that
\[
\frac{\pd}{\pd t} \phi^i(x) = - \frac{1}{\tau} f_{\ib}(\phi(x)),
\]
hence
\[
\frac{\pd}{\pd t} \frac{\pd \phi^i(x)}{\pd z^j} = - \frac{1}{\tau} f_{k \ib}(\phi(x)) \frac{\pd \phi^k }{\pd z^j}.
\]
Since we assume that $Ric$ is bounded uniformly, there is constant $\Lambda$ such that $|\n\n f| \leq \Lambda$, then the above equation implies that
\[
|\frac{\pd \phi^i(x)}{\pd z^j}| \leq \tau^\Lambda.
\]
By the above estimate and (\ref{eqn: control of distance}), we have
\[
|\hat{\w}(x, s)| \leq C e^{-s (\mu/2 + p \Lambda)}(c(n)+ r(o,x))^{\mu}
\]
for $s \in (-\infty, \infty)$, where $o$ is a minimal point of $f$, hence $\hat{\w}$ can be written as a finite linear combination

\[
\hat{\w} = \sum_{\lambda_i \leq \mu/2 + p \Lambda } a_i e^{-\lambda_i s} \theta_i.
\]
\end{proof}

Note that the map sending $\w$ to $\hat{\w}$ is injective, so the space of holomorphic $(p, 0)$-forms with polynomial growth of order at most $\mu$ is finite dimensional. So we have
\begin{cor}\label{cor: control dimension of holomorphic forms by spectrum}
Under the same assumptions as in Proposition \ref{series expansion for solutions of the f heat equation of forms}, we have 
\[
\dim_\mathbb{C} \mathcal{O}_\mu (\Lambda^{p,0}) \leq \sharp\{ 0 \leq \lambda(\Delta^d_f) \leq \frac{\mu}{2} + p \Lambda , \text{counting multiplicity}\}.
\]
\end{cor}

Next we show that dimension estimate for polynomial growth holomorphic forms can be reduced to that of holomorphic functions. 
\begin{lem}\label{lem: injectivity of interior product}
On a gradient shrinking K\"ahler Ricci soliton $(M, g, f)$ with real dimension $n=2m$, the interior product $i_{\n f}$ defines a linear map
\[
i_{\n f}:  \mathcal{O_\mu}(\Lambda^{p,0}) \to \mathcal{O}_{\mu + 1}(\Lambda^{p-1,0})  
\]
for each $p = 1,2,...$. Moreover, suppose the curvature is bounded and let $\Lambda = \sup_M |Rm| + \frac{1}{2}$, then we have 
\[
\dim Ker( i_{\n f}|_{\mathcal{O}_\mu(\Lambda^{p,0}) } ) \leq e^{c(n,p, \Lambda)(1+\mu)}.
\]
\end{lem}
\begin{proof}
Let $\w \in \mathcal{O_\mu}(\Lambda^{p,0}M)$, locally we can write
\[
\w = \frac{1}{p!} \w_{i_1 i_2... i_p} dz^{i_1} \wedge dz^{i_2} \wedge ... \wedge dz^{i_p},
\]
and 
\[
i_{\n f}\w = \frac{1}{(p-1)!} f_{\bar{i_1}}  \w_{i_1 i_2... i_p} dz^{i_2}\wedge ... \wedge dz^{i_p}.
\]
By direct calculation
\[
\n_{\kb} (f_{\bar{i_1}}  \w_{i_1 i_2... i_p}) =  f_{\bar{i_1} \kb}  \w_{i_1 i_2... i_p} + f_{\bar{i_1}}  \n_{\kb} \w_{i_1 i_2... i_p} = 0.
\]
Since $|\n f| \leq \frac{1}{2} r + c(n)$, we have 
\[
|i_{\n f} \w| \leq C (1+ r)^{\mu + 1}
\]
for some constant $C$. Thus $i_{\n f} \w \in \mathcal{O}_{\mu + 1}(\Lambda^{p-1,0}M)   $.

It remains to estimate the kernel of $i_{\n f}$ when the Ricci curvature is bounded uniformly. Let $\w \in \mathcal{O}_\mu(\Lambda^{p,0})$ such that $i_{\n f} \w = 0$. Then by $\Delta^d \w = \Delta^{\bar{\pd}} \w = 0$ and Cartan's formula, we have
\[
\Delta_f^d \w = \Delta^d \w + i_{\n f} d \w + d i_{\n f} \w = i_{\n f} d \w.
\]
Hence
\[
\begin{split}
\langle \n |\w|^2 , \n f\rangle = & \sum_k \sum_{i_1 < i_2 < ... < i_p} \n_k \w_{i_1 ... i_p} f_{\kb} \bar{\w}_{\ib_1 ... \ib_p} + \w_{i_1 ... i_p} \n_{\kb} \bar{\w}_{\ib_1...\ib_p}f_k\\
= & 2 Re \langle  i_{\n f} d\w, \w\rangle = 2 Re \langle \Delta_f^d \w, \w\rangle.  
\end{split}
\]
By Proposition \ref{series expansion for solutions of the f heat equation of forms}, there are constants $a_i$ such that 
\[
\w = \sum_{\lambda_i \leq \frac{\mu}{2} + p \Lambda} a_i \theta_i,
\]
where $\theta_i$ are the eigen-forms of $\Delta_f^d$, let $\theta_i$ be orthonormal to each other.
Thus
\[
\begin{split}
    \int_M \langle \n |\w|^2 , \n f\rangle dv_f = &\int_M 2 Re\langle \sum_{\lambda_i \leq \frac{\mu}{2} + p \Lambda} a_i \lambda_i \theta_i, \sum_{\lambda_i \leq \frac{\mu}{2} + p \Lambda} a_i \theta_i\rangle dv_f \\
    = & 2 \sum_{\lambda_i \leq \frac{\mu}{2} + p \Lambda} |a_i|^2 \lambda_i \\
    \leq &  2(\frac{\mu}{2} + p \Lambda) \int_M |\w|^2 dv_f.
\end{split}
\]
On the other hand, 
\[
\int \langle \n |\w|^2, \n f \rangle dv_f = \int |\w|^2 (|\n f|^2 - \Delta f) dv_f = \int |\w|^2 (f - \frac{n}{2}) dv_f.
\]
Therefore
\[
\int_M |\w|^2 \left( f - \frac{n}{2} - \mu - 2 p \Lambda \right) dv_f\leq 0.
\]
The rest of the argument is similar to the method for holomorphic functions in \cite{MW2014}.
Take $t_0 = \frac{n}{2}+ \mu + 2 p \Lambda + 1$ and $K = (\frac{n}{2}+ \mu + 2 p \Lambda ) e^{5t_0} $, the above inequality implies that
\[
\int_{ f < 5 t_0} |\w|^2 dv_g\leq K \int_{f < t_0} |\w|^2 dv_g. 
\]
By (\ref{eqn: estimate of the potential function}), we can choose $t_0$ larger by adding a constant depending on $n$, and take $r_0 = 2 \sqrt{t_0}$ such that $B_o(2 r_0) \subset \{f< 5 t_0\}$ and $\{f< t_0\} \subset B_o(r_0)$, where $o$ is a minimal point of $f$, hence we have
\begin{equation}\label{eqn: doubling for holomorphic form in the kernel of interior product}
\int_{ B_o(2r_0)} |\w|^2 dv_g\leq K \int_{B_o(r_0)} |\w|^2 dv_g. 
\end{equation}
The Sobolev constant on $B_o(2r_0)$ depends on $r_0$ and the Ricci lower bound in the form $r_0^2e^{c(n)\sqrt{\Lambda} r_0} / Vol(2r_0)^{2/n}$, we can apply the Moser iteration method to the differential inequality
\[
\Delta |\w|^2 \geq 2 |\n \w|^2 - C|\w|^2 , 
\]
where $C$ is a constant depending on $n,p$ and $\Lambda$, which can be derived from $\Delta^d \w = 0$ by the Weitzenb\"ock formula. Then we get a mean value inequality for $|\w|$ where the constant depends on the Sobolev constant and the curvature bound on $B_o(2r_0)$ (see for e.g. section 19 of \cite{Li2012}). Note that the volume doubling constant on $B_o(2r_0)$ is also an exponential function of $r_0$ since the curvature is bounded. Hence the method of \cite{Li1997} yields the estimate of the dimension, which is an exponential function of $r_0$ and can be written in the form
\[
dim Ker (i_{\n f}|_{\mathcal{O}_\mu(\Lambda^{p, 0}) } ) \leq e^{c(n,p,\Lambda)(1+ \mu)}.
\]
\end{proof}

\begin{prop}\label{prop: control dimension of forms by that of functions}
Let $(M, g, f)$ be a gradient K\"ahler Ricci shrinker with bounded curvature $|Rm| \leq \Lambda$, then there is a constant $C(n, p, \Lambda)$, such that 
\[
\dim  \mathcal{O_\mu}(\Lambda^{p,0})  \leq \dim \mathcal{O}_{\mu + p} + e^{C(n, p, \Lambda)(1+\mu)}. 
\]
\end{prop}

\begin{proof}
This dimension estimate follows from Lemma \ref{lem: injectivity of interior product} and induction on $p$.
\end{proof}

\begin{proof}[Proof of Theorem \ref{thm: estimate for the dimension of holomorphic forms}]
    The theorem follows from Proposition \ref{prop: control dimension of forms by that of functions} and Corollary \ref{cor: control dimension of holomorphic forms by spectrum}.
\end{proof}

As a separate issue, we also obtain a Liouville type theorem for eigen-$1$-forms of the $f$-Hodge Laplacian on gradient shrinking Ricci solitons, which may be of independent interest. The case for $f$-harmonic $1$-forms when $M$ is closed was known in \cite{Lott2003}.
\begin{prop}\label{prop: vanishing of L2 f-harmonic 1-forms}
Let $\eta$ be a $1$-form on a gradient shrinking Ricci soliton $(M, g, f)$, suppose $\eta \in L^2(dv_f)$ and $\Delta^d_f \eta = \lambda \eta$ with $\lambda < \frac{1}{2}$, then $\eta = 0$.
\end{prop}
\begin{proof}
We first compute the Lie derivative
\[
\begin{split}
\mathcal{L}_{\n f} \eta = & ( i_{\n f} d  + d i_{\n f} )\eta =  \left(  \n_i \eta_j f_i + f_{ij} \eta_i \right) dx^j, 
\end{split}
\]
hence the $f$-Hodge Laplacian of $\eta$ is given by
\begin{equation}\label{eqn: Holdge Laplacian of a 1-form}
\Delta_f^d \eta = (\Delta^d + \mathcal{L}_{\n f}) \eta = - \Delta \eta + Ric_f(\eta) +   \n_i \eta_j  f_i  dx^j.
\end{equation}
Hence
\[
\begin{split}
\Delta_f |\eta|^2 = & 2 \langle \Delta \eta, \eta \rangle + 2 |\n \eta|^2 - \langle \n f, \n |\eta|^2 \rangle  \\
= & -2\lambda |\eta|^2 + 2 Ric_f(\eta, \eta) + 2 \n_i \eta_j f_i \eta_j  + 2 |\n \eta|^2 - 2 f_i \n_i \eta_j \eta_j  \\
= & (1-2\lambda)|\eta|^2 + 2 |\n \eta|^2,
\end{split}
\]
from which we derive
\[
\Delta_f |\eta| \geq (\frac{1}{2} - \lambda)|\eta|. 
\]
Then the result follows from the fact that nonnegative $f$-subharmonic functions in  $L^2(dv_f)$ are constants. Indeed, let $u$ be such a function, let $\phi$ be a cut-off function on a geodesic ball $B(r)$ with radius $r$, such that $|\n \phi| \leq \frac{4}{r} $, then
\[
0 \leq \int \phi^2 u \Delta_f u dv_f = \int - \phi^2 |\n u|^2 - 2 \phi u \langle \n \phi, \n u \rangle \leq \int - \frac{1}{2}\phi^2 |\n u|^2 + 4 |\n \phi|^2 u^2 dv_f,
\]
let $r \to \infty$ yields that $u = constant$.

Now since $|\eta| = constant$ and $\lambda < \frac{1}{2}$,  the differential inequality $\Delta_f |\eta| \geq (\frac{1}{2} -\lambda) |\eta|$ implies that $\eta = 0$. 
\end{proof}
As a consequence, the first eigenvalue of $\Delta_f^d$ on $1$-forms is at least $\frac{1}{2}$ on a gradient Ricci shrinker, and the result in Corollary \ref{cor: control dimension of holomorphic forms by spectrum} for $(1,0)$-forms can be improved.
\begin{cor}\label{cor: control dimension of holomorphic (1,0) forms by spectrum}
Under the same assumptions as in Proposition \ref{series expansion for solutions of the f heat equation of forms}, we have 
\[
\dim_\mathbb{C} \mathcal{O}_\mu (\Lambda^{1,0}) \leq \sharp\{ \frac{1}{2} \leq \lambda(\Delta^d_f) \leq \frac{\mu}{2} + \Lambda , \text{counting multiplicity}\}.
\]
\end{cor}


\end{document}